\documentclass[11pt]{amsart}
\usepackage{amsmath,amsthm,amssymb,color,enumitem,fullpage,float,hyperref,amsfonts,graphics,tcolorbox,tikz,verbatim, tikz-cd}
\setlist{nolistsep}
\DeclareMathOperator{\diag}{diag}
\DeclareMathOperator{\im}{im}

\DeclareMathOperator{\Arith}{Arith}
\DeclareMathOperator{\coker}{coker}

\newcommand{\Cballot}{A}

\newcommand{\Nn}{{\mathbb N}}
\newcommand{\Z}{{\mathbb Z}}
\newcommand{\Zz}{{\mathbb Z}}

\newcommand{\C}{{\mathcal C}}

\renewcommand{\P}{{\mathcal P}}

\newcommand{\dd}{\mathbf{d}}
\newcommand{\rr}{\mathbf{r}}
\newcommand{\1}{\mathbf{1}}
\newcommand{\2}{\mathbf{2}}

\newcommand{\caseif}{&\text{ if }}
\newcommand{\multiset}[2]{\ensuremath{\left(\kern-.3em\left(\genfrac{}{}{0pt}{}{#1}{#2}\right)\kern-.3em\right)}} 
\DeclareMathOperator{\Multiset}{MSet} 

\newcommand{\st}{\mid} 

\newtheorem{theorem}{Theorem}
\newtheorem{lemma}[theorem]{Lemma}
\newtheorem{prop}[theorem]{Proposition}
\newtheorem{cor}[theorem]{Corollary}

\newtheorem{Claim}[theorem]{Claim}
\theoremstyle{definition}

\newtheorem{Remark}[theorem]{Remark}
\newtheorem{Example}[theorem]{Example}

\newcommand\commentout[1]{}

\author{Benjamin Braun}
\address[B. Braun]{715 Patterson Office Tower, University of Kentucky, Lexington, KY 40506, USA}
\email{benjamin.braun@uky.edu}

\author{Hugo Corrales}
\address[H. Corrales]{Departamento de Matem\'aticas, Centro de Investigaci\'on y de Estudios Avanzados del IPN, Apartado Postal 14--740, 07000 Ciudad de M\'exico, M\'exico
}
\email{hhcorrales@gmail.com}

\author{Scott Corry}
\address[S. Corry]{Department of Mathematics, Lawrence University, Appleton, WI 54911, USA}
\email{corrys@lawrence.edu}

\author{Luis David Garc\'{\i}a Puente}
\address[L. D. Garc\'{\i}a Puente]{Department of Mathematics and Statistics, Sam Houston State University, Huntsville, TX 77341-2206, USA}
\email{lgarcia@shsu.edu}

\author{Darren Glass}
\address[D. Glass]{Gettysburg College, 300 N. Washington St, Gettysburg, PA 17325, USA}
\email{dglass@gettysburg.edu}

\author{Nathan Kaplan}
\address[N. Kaplan]{Department of Mathematics, University of California, Irvine, CA 92697, USA}
\email{nckaplan@math.uci.edu}

\author{Jeremy L.\ Martin}
\address[J. L. Martin]{Department of Mathematics, University of Kansas, Lawrence, KS 66045-7594, USA}
\email{jlmartin@ku.edu}

\author{Gregg Musiker}
\address[G. Musiker]{School of Mathematics, University of Minnesota, Minneapolis, MN 55455, USA}
\email{musiker@math.umn.edu}

\author{Carlos E. Valencia}
\address[C. E. Valencia]{Departamento de Matem\'aticas, Centro de Investigaci\'on y de Estudios Avanzados del IPN, Apartado Postal 14--740, 07000 Ciudad de M\'exico, M\'exico
}
\email{cvalencia@math.cinvestav.edu.mx}

\title{Counting Arithmetical Structures on Paths and Cycles}
\date{21 June 2018}


\thanks{%
BB was supported by National Security Agency Grant H98230-16-1-0045.
LDGP was supported by Simons Collaboration Grant \#282241.
NK was partially supported by an AMS-Simons Travel Grant and by NSA Young Investigator Grant H98230-16-10305.
JLM was supported by Simons Collaboration Grant \#315347.
GM was supported by NSF Grant \#13692980.
CEV was partially supported by SNI}

\begin{document}
\maketitle

\begin{abstract}
Let $G$ be a finite, connected graph.  An arithmetical structure on $G$ is a pair of positive integer vectors $\dd,\rr$ such that $(\diag(\dd)-A)\rr=0$, where $A$ is the adjacency matrix of $G$.
We investigate the combinatorics of arithmetical structures on path and cycle graphs, as well as the associated critical groups (the torsion part of the cokernels of the matrices $(\diag(\dd)-A)$).
For paths, we prove that arithmetical structures are enumerated by the Catalan numbers, and we obtain refined enumeration results related to ballot sequences.
For cycles, we prove that arithmetical structures are enumerated by the binomial coefficients $\binom{2n-1}{n-1}$, and we obtain refined enumeration results related to multisets.
In addition, we determine the critical groups for all arithmetical structures on paths and cycles.
\end{abstract}


\section{Introduction}\label{sec:intro}

This paper is about the combinatorics of arithmetical structures on path and cycle graphs.  We begin by recalling some basic facts about graphs, Laplacians, and critical groups.

Let $G$ be a finite, connected graph with $n\geq 2$ vertices, let $A$ be its adjacency matrix, and let $D$ be the diagonal matrix of vertex degrees.  The \emph{Laplacian matrix} $L=D-A$ has rank $n-1$, with nullspace spanned by the all-ones vector $\1$.  If we regard $L$ as a $\Zz$-linear transformation $\Zz^n\to\Zz^n$, the cokernel $\Zz^n/\im L$ has the form $\Zz\oplus K(G)$; here $K(G)$, the \emph{critical group} is finite abelian, with cardinality equal to the number of spanning trees of $G$, by the Matrix-Tree Theorem.  The critical group is also known as the \emph{sandpile group} or the \emph{Jacobian}. The elements of the critical group represent long-term behaviors of the well-studied \emph{abelian sandpile model} on $G$; see, e.g., \cite{BTW,Dhar,WhatIs}.

More generally, an \emph{arithmetical structure} on $G$ is a pair $(\dd, \rr)$ of positive integer vectors such that $\rr$ is primitive (the gcd of its coefficients is $1$) and
\[
(\diag(\dd)-A)\rr={\bf 0}.
\]
This definition generalizes the Laplacian arithmetical structure just described, where $\dd$ is the vector of vertex degrees and $\rr=\1$.  Note that each of $\dd$ and $\rr$ determines the other uniquely, so we may regard any of $\dd$, $\rr$, or the pair $(\dd,\rr)$ as an arithmetical structure on $G$.  Where appropriate, we will use the terms \emph{arithmetical $d$-structure} and \emph{arithmetical $r$-structure} to avoid ambiguity.
The set of all arithmetical structures on $G$ is denoted $\Arith(G)$, and the data $G,\dd,\rr$ together determine an \emph{arithmetical graph}.  As in the classical case, the matrix $L(G,\dd)=\diag(\dd)-A$ has rank $n-1$ \cite[Proposition~1.1]{Lorenzini89}.  The torsion part of $\coker L$ is the \emph{critical group} of the arithmetical graph.

Arithmetical graphs were introduced by Lorenzini in~\cite{Lorenzini89} to model degenerations of curves.  Specifically, the vertices of $G$ represent components of a degeneration of a given curve, edges represent intersections of components, and the entries of $\dd$ are self-intersection numbers.  The critical group is then the group of components of the N\'eron model of the Jacobian of the generic curve (an observation attributed by Lorenzini to Raynaud).  In this paper, we will not consider the geometric motivation, but instead study arithmetical graphs from a purely combinatorial point of view.

It is known~\cite[Lemma~1.6]{Lorenzini89} that $\Arith(G)$ is finite for all connected graphs~$G$.  The proof of this fact is non-constructive (by reduction to Dickson's lemma), raising the question of enumerating arithmetical structures for a particular graph or family of graphs. We will see that when $G$ is a path or a cycle, the enumeration of arithmetical structures on $G$ is controlled by the combinatorics of Catalan numbers.  In brief, the path $\P_n$ and the cycle $\C_n$ on $n$ vertices satisfy
\[|\Arith(\P_n)|=C_{n-1}=\frac{1}{n}\binom{2n-2}{n-1},\qquad\qquad
|\Arith(\C_n)|=\binom{2n-1}{n-1}=(2n-1) C_{n-1}\]
(Theorems~\ref{path-count} and~\ref{main-cycle-theorem}, respectively).
These results were announced in \cite{arithmetical}.

We will refine these results, and show for example that the number of $\dd$-structures on $\P_n$ with one prescribed $d_i$ entry are given by the \emph{ballot numbers}, a well-known combinatorial refinement of the Catalan numbers first investigated by Carlitz~\cite{ballot}.  For cycles, we get a similar result where the ballot numbers are replaced by binomial coefficients.  The critical group of an arithmetical structure on a path is always trivial, while for a cycle it is always cyclic of order equal to the number of occurrences of 1 in the associated arithmetical $r$-structure. Our approaches for these two families are similar: ballot sequences yield information about arithmetical structures for paths, while multisets produce information in the case of cycles.  Our main results for paths and cycles mirror each other, as do the proof techniques we use.

Two graph operations that play a central role in our work are \emph{subdivision} (or \emph{blowup}) and \emph{smoothing}.  On the level of graphs, subdividing an edge inserts a new degree-2 vertex between its endpoints, while smoothing a vertex of degree~2 removes the vertex and replaces its two incident edges with a single edge between the adjacent vertices. These operations extend to arithmetical structures and preserve the critical group, as shown by Corrales and Valencia \cite[Thms.~5.1, 5.3, 6.5]{arithmetical}, following Lorenzini~\cite[pp.484--485]{Lorenzini89}.  These operations turn out to be key in enumerating arithmetical structures.  Note that paths and cycles are special because they are precisely the connected graphs of maximum degree 2, hence can be obtained from very small graphs by repeated subdivision.

Looking ahead, many open questions remain about arithmetical graphs.  It is natural to ask to what extent the arithmetical critical group $K(G,\dd,\rr)$ behaves like the standard critical group.  The matrix $L(G,\dd)$ is an \emph{M-matrix} in the sense of numerical analysis (see, e.g., Plemmons~\cite{Plemmons}), so it admits a generalized version of chip-firing as described by Guzm\'{a}n and Klivans \cite{GuzmanKlivans}.  One could also look for an analogue of the matrix-tree theorem, asserting that the cardinality of the critical group enumerates some tree-like structures on the corresponding arithmetical graphs, or for a version of Dhar's burning algorithm~\cite{Dhar} that gives a bijection between those structures and objects like parking functions.

Enumerating arithmetical structures for graphs other than paths and cycles appears to be more difficult.  For example, the arithmetical d-structures on the star $K_{n,1}$ can be shown to be the positive integer solutions to the equation
\[d_0 = \sum_{i=1}^n \frac{1}{d_i}.\]
A solution to this Diophantine equation is often called an \emph{Egyptian fraction representation} of $d_0$.  The numbers of solutions for $n\leq 8$ are given by sequence \href{http://oeis.org/A280517}{A280517} in~\cite{OEIS}.  A related problem, with the additional constraints $d_0=1$ and $d_1\leq\cdots\leq d_n$, was studied by S\'andor~\cite{sandor}, who gave upper and lower bounds for the number of solutions; the upper bound was subsequently improved by Browning and Elsholtz \cite{browning_elsholtz}.  The lower and upper bounds are far apart, and it is unclear even what asymptotic growth to expect.

{\bf Acknowledgments:} This project began at the ``Sandpile Groups'' workshop at Casa Matem\'{a}tica Oaxaca (CMO) in November 2015, funded by Mexico's Consejo Nacional de Ciencia y Tecnolog\'{i}a (CONACYT).  The authors thank CMO and CONACYT for their hospitality, as well as Carlos A.\ Alfaro, Lionel Levine, Hiram H.\ Lopez, and Criel Merino for helpful discussions and suggestions. The authors also thank the anonymous referees for their helpful suggestions.


\section{Paths}\label{sec:paths}

We have two main goals in this section.
First, we show in Theorem~\ref{main-path-theorem} that using $r$-structures one can partition the arithmetical structures on a fixed path into sets with cardinality given by ballot numbers, generalizing Theorem~\ref{path-count} which states that the total number of arithmetical structures is given by a Catalan number.
Second, we show in Theorem~\ref{ballot-path} that using $d$-structures one can produce additional partitions of the arithmetical structures of a fixed path that again has distribution given by ballot numbers.
We begin with some basic results about arithmetical structures on paths.


\begin{lemma}\label{path-lemma}
If $\rr = (r_1,\ldots, r_n)$ is an arithmetical $r$-structure on the path $\P_n$ with $n \ge 2$ vertices, then
\[
r_1 = r_n = 1.
\]
Moreover, if $r_j = 1$ for some $1 < j < n$, then $(r_1,\ldots, r_j)$ is an arithmetical $r$-structure on $\P_j$ and $(r_j,\ldots, r_n)$ is an arithmetical $r$-structure on $\P_{n-(j-1)}$.
\end{lemma}

This result follows from~\cite[Theorem 4.2]{arithmetical}, but for the sake of illustrating typical methods, we include a short self-contained proof.

\begin{proof}
First, note that $(\dd,\rr)$ is an arithmetical structure on $\P_n$ if and only if the following equalities hold:
\begin{equation} \label{path-d-r-equalities}
\begin{aligned}
r_1d_1&=r_2;\\
r_id_i&=r_{i-1}+r_{i+1}\quad\text{ for }1<i<n;\quad\text{ and }\\
r_nd_n&=r_{n-1}.
\end{aligned}
\end{equation}
Starting with the first equation and moving down, we obtain the sequence of divisibilities:
\[
r_1|r_2 \implies r_1|r_3 \implies \dotsm \implies r_1|r_n.
\]
Since $\rr$ is a primitive vector, we conclude that $r_1=1$. The same argument starting with the last equation and moving up yields $r_n=1$.

On the other hand, if $r_j=1$ for some $1<j<n$, then the following slight modification of the first $j$ equations defining $(\dd,\rr)$ shows that $(r_1,\dots,r_j)$ is an arithmetical structure on $\P_j$:
\begin{align*}
d_1&=r_2 && (\textrm{since $r_1=1$ by part (1)}),\\
r_id_i&=r_{i-1}+r_{i+1} \quad\text{ for }1<i<j;\quad\text{ and }\\
\tilde{d}_j&:=r_{j-1}.
\end{align*}
A similar argument, using the final $n-(j-1)$ equations defining the pair $(\dd,\rr)$, shows that $(r_j,\dots,r_n)$ is an arithmetical structure on $\P_{n-(j-1)}$.
\end{proof}

\begin{cor} \label{path-characterization}
Let $\rr=(r_1,\dots,r_n)$ be a primitive positive integer vector.  Then
$\rr$ is an arithmetical $r$-structure on $\P_n$ if and only if 
\begin{enumerate}[label=(\alph*)]
\item $r_1=r_n=1$, and \\
\item $r_i|(r_{i-1}+r_{i+1})$ for all $i\in[2,n-1]$.
\end{enumerate}
\end{cor}
\begin{proof}
Condition (a) is part of Lemma~\ref{path-lemma}, and the necessity of condition (b) follows from~\eqref{path-d-r-equalities}.  On the other hand, if $\rr$ satisfies these two conditions then the corresponding $d$-structure can be recovered from the equations given in~\eqref{path-d-r-equalities}.
\end{proof}

For an arithmetical $r$-structure $\rr$, let
\[\rr(1) = \#\{i \st r_i=1\}.\]
\begin{theorem} \label{path-count}
The number of arithmetical structures on $\P_n$ is the Catalan number $C_{n-1}=\frac{1}{n}\binom{2n-2}{n-1}$.  Moreover, the number of arithmetical $r$-structures with $\rr(1) = 2$ is the Catalan number $C_{n-2}$.
\end{theorem}

\begin{proof}
For the second assertion, the description in Corollary~\ref{path-characterization} is a known interpretation of the Catalan numbers; see \cite{AignerSchulze} or \cite[p.~34, Problem~92]{catalan}.  The first assertion then follows from the standard Catalan recurrence $C_{n-1} = \sum_{i=0}^{n-2}C_iC_{n-2-i}$, since by Lemma~\ref{path-lemma} the same recurrence holds for arithmetical structures.
\end{proof}

As a consequence of Theorem~\ref{path-count}, the path $\P_2$ has only one arithmetical structure, namely $(\dd,\rr)=(\1,\1)$, and it is the only path with a unique arithmetical structure.
In the rest of this section, we study a finer enumeration of arithmetical structures on paths in terms of their $\rr$-vectors (Theorem~\ref{main-path-theorem}) and $\dd$-vectors (Theorem~\ref{ballot-path}). The arithmetical structure on $\P_n$ that comes from the Laplacian of $\P_n$ has $\rr = (1,\ldots, 1)$ and $\dd = (1,2,2,\ldots, 2,1)$.  We call this pair $(\dd,\rr)$ the \emph{Laplacian arithmetical structure}.

We recall the following result of Corrales and Valencia.
\begin{prop}[{\cite[Thm.~6.1]{arithmetical}}]\label{only-two-path}
There is exactly one arithmetical structure $(\dd,\rr)$ on $\P_n$ such that $d_i\geq 2$ for all $1<i<n$, namely the Laplacian arithmetical structure.
\end{prop}

Given a graph $G$, the \emph{subdivision} of an edge $e = uv$ with endpoints $u$ and $v$ yields a graph containing one new vertex $w$, and with a pair of edges $uw$ and $vw$ replacing $e$.
The reverse operation, known as \emph{smoothing} a 2-valent vertex $w$ incident to edges $e_1 = uw$ and  $e_2 = wv$, removes both edges $e_1, e_2$ and adds a new edge connecting $u$ and $v$.

One of our starting points is the following result relating the arithmetical structures on $\P_n$ with the arithmetical structures on $\P_{n+1}$, regarding the latter graph as an edge subdivision of the former.  The construction is a particular case of the \emph{blow-up} operation described in \cite[pp.~484--485]{Lorenzini89} as well as of the \emph{clique-star transformation}~\cite[Theorem 5.1]{arithmetical}. The recursive structure of the $\rr$-vector was also described in \cite[Lemma~2]{AignerSchulze}.  We include a unified proof of these facts.

\begin{prop}
\label{blowup-path}
(a) Let $n\geq 2$ and let $(\dd',\rr')\in\Arith(\P_n)$. Given $i$ with $2\leq i \leq n$, define integer vectors $\dd$ and $\rr$ of length~$n+1$ as follows:

\begin{equation} \label{path-subdivision}
d_j = \begin{cases}
d'_j \caseif j < i-1,\\
d'_{i-1}+1 \caseif j=i-1,\\
1 \caseif j=i,\\
d'_{i}+1 \caseif j=i+1,\\
d'_{j-1} \caseif j>i+1,
\end{cases}
\qquad\qquad
r_j = \begin{cases}
r'_j \caseif j<i,\\
r'_{i-1}+r'_i \caseif j=i,\\
r'_{j-1} \caseif j>i,
\end{cases}
\end{equation}
for $1\leq j\leq n+1$.  Then $(\dd,\rr)$ is an arithmetical structure on $\P_{n+1}$.
Moreover, the cokernels of $L(\P_n,\dd')$ and $L(\P_{n+1},\dd)$ are isomorphic.

(b)
Let $n\geq 3$ and let $(\dd,\rr)\in\Arith(\P_n)$ such that $d_i=1$ for some $1<i<n$.
For $j\in[n-1]$, define
\begin{equation} \label{path-smoothing}
d'_j = \begin{cases} d_j \caseif j<i-1, \\ d_{i-1}-1 \caseif j=i-1,\\ d_{i+1}-1 \caseif j=i,\\ d_{j+1}\caseif i<j\le n-1,\end{cases}\qquad\qquad
r'_j = \begin{cases} r_j \caseif j<i,\\ r_{j+1}\caseif j\ge i.\end{cases}\qquad\qquad
\end{equation}

Then $(\dd',\rr')$ is an arithmetical structure on $\P_{n-1}$.
Again, the cokernels of $L(\P_n,\dd)$ and $L(\P_{n-1},\dd')$ are isomorphic.
\end{prop}

In case (a), we say that $(\dd,\rr)$ is the \emph{subdivision} of $(\dd',\rr')$ at position~$i$.  In case~(b), we say that $(\dd',\rr')$ is the \emph{smoothing} of $(\dd,\rr)$ at position $i$.

\begin{proof}
Given a graph $G$ and a clique $C$ in $G$, the \emph{clique-star transformation} removes every edge in $C$ and adds a new vertex $w$ together with all the edges between $w$ and every vertex in $C$. Clearly, given an edge $e$ in $\P_n$, the clique-star transformation at $e$ is precisely the subdivision of $e$.
It follows as a special case of \cite[Theorem 5.1]{arithmetical} that equation~\eqref{path-subdivision} gives an arithmetical structure on $\P_{n+1}$.
On the other hand, it is clear that if $(\dd',\rr')$ is the smoothing of $(\dd,\rr)$ at position $i$ then $(\dd,\rr)$ is the subdivision of $(\dd',\rr')$ at position $i$.  Hence equation~\eqref{path-smoothing} gives an arithmetical structure on $\P_{n-1}$.

It remains to show that smoothing is well-defined, that is, if $d_i=1$ for some $1<i<n$, then $d_{i-1} \geq 2$ and $d_{i+1} \geq 2$.
In other words, we need to show that if $(\dd,\rr)$ is an arithmetical structure on $\P_n$ with $n\geq 3$, then there are no consecutive 1's in $\dd$. This follows immediately from Lemma~\ref{isolated-ones-in-d} below, which applies to arithmetical structures on arbitrary graphs.

Finally, the fact that the corresponding cokernels are isomorphic follows directly from \cite[\S 1.8]{Lorenzini89}. Explicitly, let $(\dd,\rr)$ be the \emph{subdivision} of $(\dd',\rr')$ at position~$i$, let $M'=L(\P_{n},\dd')$, and let $M=L(\P_{n+1},\dd)$. Then $M$ is $\Zz$-equivalent to the matrix
\[
M_Q = \begin{pmatrix}M'+Q^tQ & -Q^t\\ -Q & 1\end{pmatrix},
\]
where $Q$ is the vector of length $n$ with 1's in the two positions $i-1$ and $i$, and 0's elsewhere.  (Here ``$\Zz$-equivalent'' means $M=AM_QB$ where $A$ and $B$ are invertible over $\Zz$.)
Moreover, $M_Q$ is $\Zz$-equivalent to the matrix $\begin{pmatrix}
M' & 0 \\ 0 & 1\end{pmatrix}$. Therefore, the cokernel of $M$ is isomorphic to the cokernel of~$M'$.
\end{proof}

\begin{lemma}\label{isolated-ones-in-d}
Let $G$ be a connected graph on the vertex set $\{v_1,v_2,\dots,v_n\}$ with $n\ge 3$ and adjacency matrix $A=(a_{ij})$. In addition, suppose that $v_1$ and $v_2$ are neighbors (that is, $a_{12}>0$). If $\dd$ is an arithmetical $d$-structure on $G$ with $d_1=1$, then $d_2>1$.
\end{lemma}

\begin{proof}
Let $\rr=(r_1,r_2,\dots,r_n)$ denote the $r$-structure corresponding to the $d$-structure $\dd$. By the definition of arithmetical structure, we have
\[
d_1r_1=\sum_{i=1}^na_{1i}r_i \qquad \textrm{and} \qquad
d_2r_2=\sum_{j=1}^na_{2j}r_j.
\]
Now suppose, contrary to the claim, that $d_1=d_2=1$. Since $n\ge 3$, at least one of the vertices $v_1,v_2$ has another neighbor; without loss of generality, suppose that $v_2$ and $v_3$ are connected by an edge of $G$, so that $a_{23}>0$. Then
\begin{eqnarray*}
r_1r_2&=&d_1r_1d_2r_2\\
&=&\left(\sum_{i=1}^na_{1i}r_i\right)\left(\sum_{j=1}^na_{2j}r_j\right)\\
&=&a_{12}^2r_1r_2+a_{12}a_{23}r_2r_3+\sum_{\substack{i,j=1 \\ (i,j)\ne (2,1),(2,3)}}^na_{1i}a_{2j}r_ir_j.
\end{eqnarray*}
Subtracting $r_1r_2$ yields the equation
\begin{eqnarray*}
0=(a_{12}^2-1)r_1r_2+a_{12}a_{23}r_2r_3+
\sum_{\substack{i,j=1 \\ (i,j)\ne (2,1), (2,3)}}^na_{1i}a_{2j}r_ir_j.
\end{eqnarray*}
But this is impossible, since for all $k,\ell\in[n]$, the quantities $a_{k\ell}\ge 0, r_k\ge 1$, and $a_{12},a_{23}>0$ by assumption.
\end{proof}

\begin{theorem}\label{critical-group-path}
Let $(\dd,\rr)\in\Arith(\P_n)$ be an arithmetical structure on the path.  Then the associated critical group $K(\P_n,\dd,\rr)$ is trivial.  Moreover,
\begin{equation} \label{r-and-d-path}
\rr(1)=3n-2-\sum_{j=1}^n d_j.
\end{equation}
\end{theorem}

\begin{proof}
We proceed by induction on $n$.  It is easy to check that the theorem holds for the base case $n=2$, where the only arithmetical structure is the Laplacian arithmetical structure $(\dd,\rr)=(\1,\1)$.

For any $n\ge 3$, if $\dd=(1,2,2,\dots,2,1)$ and $\rr=\1$, then $3n-2-\sum_{i=1}^n d_i=n=\rr(1)$.  In this case $L$ is the Laplacian, and $K(\P_n,\dd,\1)$ is the standard critical group of $\P_n$, which is trivial (since $\P_n$ has only one spanning tree, namely itself).

Now suppose that $n\geq3$ and $\dd\neq(1,2,2,\dots,2,1)$.  By Proposition~\ref{only-two-path}, $d_i=1$ for some $1<i<n$. Proposition~\ref{blowup-path} implies that
$(\dd,\rr)$ can be obtained from subdividing some $(\dd',\rr')\in\Arith(\P_{n-1})$ at position~$i$, and $K(\P_n,\dd,\rr)=K(\P_{n-1},\dd',\rr')$.  By induction, these groups are trivial.  Moreover, we note that $d_ir_i=r_{i-1}+r_{i+1} \ge 2$, which implies that $r_i>1$.  We therefore have
\begin{align*}
\rr(1) &= \rr'(1) = 3(n-1) - 2 - \sum_{j=1}^{n-1} d'_j && \text{(by induction)}\\
&= 3(n-1) - 2 - \left(\sum_{j=1}^n d_j - (2+d_i)\right)\\
&= 3n - 2 - \sum_{j=1}^n d_j && \text{(because $d_i=1$).}\qedhere
\end{align*}
\end{proof}

We next consider a finer enumeration of arithmetical $r$-structures on paths, based on the value of~$\rr(1)$.

A \emph{lattice path} is a walk from $(0,0)$ to $(p,q)\in\Nn^2$ consisting of $p$ east steps and $q$ north steps.  It is a standard fact that the number of such paths is $\binom{p+q}{p}$, and that the number of paths from $(0,0)$ to $(k,k)$ that do not cross above the line $y=x$ is the Catalan number $C_k$.  More generally, let $B(k,\ell)$ denote the number of lattice paths from $(0,0)$ to $(k,\ell)$ that do not cross above the line $y=x$.  The numbers $B(k,\ell)$ are a generalization of the Catalan numbers known as \emph{ballot numbers}, sequence \href{https://oeis.org/A009766}{A009766} in the On-Line Encyclopedia of Integer Sequences \cite{OEIS}. Ballot numbers were first observed by Bertrand in 1887~\cite{Bertrand} and studied in detail by Carlitz~\cite{ballot}.  Our $B(k,\ell)$ corresponds to $f(k+1,\ell+1)$ in Carlitz's notation.  A formula is given by $B(k,\ell)=\frac{k-\ell+1}{k+1}\binom{k+\ell}{k}$ \cite[eqn.~2.12]{ballot}.

 \begin{Remark} \label{Ballot-Triangulation}
According to Pak's historical survey in the appendix to \cite{catalan}, the first appearance of the Catalan numbers $C_k$ in the literature was as the number of triangulations of a $(k+2)$-gon, as shown by Segner \cite{Segner} and Euler \cite{Euler}.  In this context, the ballot number $B(k,\ell)$ counts the number of triangulations of a $(k+3)$-gon with a distinguished vertex that has $k-\ell+1$ triangles incident to it.
\end{Remark}

\begin{theorem} \label{main-path-theorem}
Let $n\geq 2$ and $1\leq k\leq n$, and let $\Cballot(n,k)$ be the number of arithmetical structures $(\dd,\rr)$ on $\P_n$ such that $\rr(1)=k$.  Then
\begin{equation} \label{main-path-thm-eqn}
\Cballot(n,k) = B(n-2,n-k) = \frac{k-1}{n-1} \binom{2n-2-k}{n-2}.
\end{equation}
\end{theorem}

Combining Theorems~\ref{critical-group-path} and~\ref{main-path-theorem} immediately yields the following result.
\begin{cor}\label{cor:10}
For $n\geq 2$ and any $k$, we have
\[\#\left\{(\dd,\rr)\in\Arith(\P_n) \st \sum_{j=1}^nd_j=k\right\} = B(n-2,k-2n+2).
\]
In particular, there are no arithmetical structures with $\sum_{j=1}^nd_j=k$ unless $2n-2\le k\le 3n-4$.
\end{cor}

We give two proofs of Theorem \ref{main-path-theorem}.  The first is bijective and involves identifying lattice paths with \emph{ballot sequences} and interpreting arithmetical structures on $\P_n$ in terms of sequences of edge subdivisions of arithmetical structures on shorter paths. The second proof is phrased in terms of recurrences for certain classes of lattice paths.

To set up the first proof, we describe how to obtain an arithmetical structure on $\P_n$ starting from an arithmetical structure on a path $\P_m$ with $m<n$ and repeatedly subdividing edges in an order specified by an integer sequence $\mathbf{b} = (b_1,\ldots, b_{n-m})$.  We first note the following result, which can be found in \cite[Thm.~6.1]{arithmetical}, but also follows as an immediate consequence of Propositions \ref{only-two-path} and \ref{blowup-path}.

\begin{prop}\label{subdivision-prop}
Any arithmetical structure $(\dd,\rr)$ on $\P_n$ not equal to the Laplacian arithmetical structure can be obtained from an arithmetical structure on $\P_{n-1}$ by subdividing an edge.
\end{prop}

This result implies that every arithmetical structure on $\P_n$ comes from taking the Laplacian arithmetical structure on $\P_m$ for some $m$ satisfying $2\le m\le n$ and subdividing $n-m$ edges.   Consider the standard ordering on the $m-1$ edges of $\P_m$. Let $\mathbf{b} = (b_1,\ldots, b_{n-m})$ be a sequence of $n-m$ positive integers satisfying $1 \le b_i \le (m-1) + (i-1)=m+i-2$.  We inductively define an arithmetical structure $\mathbf{A}_n(\mathbf{b})$ on $\P_n$ from this sequence $\mathbf{b}$.  Let $(\dd_0, \rr_0)$ be the Laplacian arithmetical structure on $\P_m$. Let $(\dd_i, \rr_i)$ be the arithmetical structure on $\P_{m+i}$ obtained from the arithmetical structure $(\dd_{i-1}, \rr_{i-1})$ by subdividing the edge $b_i$ of the path $\P_{m+i-1}$; equivalently, the position of the vector specified by the index $i$ in part (a) of Proposition~\ref{blowup-path} is given by $b_i+1$. Proposition~\ref{blowup-path} gives explicit formulas for the entries of $(\dd_i, \rr_i)$.  Note that if $n=m$, then $\mathbf{b}$ is the empty sequence and $\mathbf{A}_n(\mathbf{b})$ is the Laplacian arithmetical structure on $\P_n$.  Proposition \ref{blowup-path} implies that the value of $\rr(1)$ is preserved under edge subdivision, so if $\mathbf{b} = (b_1,\ldots, b_{n-m})$ then $\mathbf{A}_n(\mathbf{b})$ is an arithmetical structure on $\P_n$ with $\rr(1) = m$.

\begin{Example}
Let $n=5$ and $m=2$, with $\mathbf{b}=(1,2,2)$.
Then our subdivision process on arithmetical $d$-structures yields
\[
(1,1)\mapsto (2,1,2)\mapsto (2,2,1,3) \mapsto (2,3,1,2,3) \, ,
\]
and thus $\mathbf{A}_5(1,2,2)=(2,3,1,2,3)$.
The corresponding arithmetical $r$-structure transformation is
\[
(1,1)\mapsto (1,2,1)\mapsto (1,2,3,1)\mapsto (1,2,5,3,1)\, .
\]
\end{Example}

Two different sequences can produce the same arithmetical structure on $\P_n$, but it is straightforward to characterize when this occurs.
\begin{lemma}\label{lemma_bseq}
Let $n \ge m \ge 2$ and $\mathbf{b} = (b_1,b_2,\ldots, b_{n-m})$ be a sequence of positive integers satisfying $1 \le b_i \le i+m-2$.  Suppose $i$ is a positive integer satisfying $1 \le i < n-m$ with $b_i > b_{i+1}$.  Then, $\mathbf{A}_n(\mathbf{b}) = \mathbf{A}_n(\mathbf{b'})$, where $\mathbf{b}' = (b_1',\ldots, b_{n-m}')$ is defined by
\[
b_j' = \begin{cases}
b_{i+1} & \text{if } j=i,\\
b_i + 1 & \text{if } j = i+1,\\
b_j & \text{otherwise}.
\end{cases}
\]
\end{lemma}
\begin{proof}
By the definition of $(\dd_i, \rr_i)$ and Proposition \ref{blowup-path} we see that if $b_i > b_{i+1}$, then starting with the arithmetical structure $(\dd_{i-1}, \rr_{i-1})$, subdividing $\P_{m+i-1}$ at edge $b_i$ and then subdividing $\P_{m+i}$ at edge $b_{i+1}$, gives the same arithmetical structure on $\P_{m+i+1}$ as subdividing $\P_{m+i-1}$ at edge $b_{i+1}$ and then subdividing $\P_{m+i}$ at edge $b_i + 1$.
\end{proof}

Repeatedly applying this lemma gives a bijection between arithmetical structures on $\P_n$ and sequences of a certain type.

\begin{prop}\label{prop_bseq}
Every arithmetical structure on $\P_n$ is equal to $\mathbf{A}_n(\mathbf{b})$ for a unique sequence $\mathbf{b} = (b_1,\ldots, b_{n-m})$ satisfying $1\le b_i \le i+m - 2$ and $b_i \le b_{i+1}$ for all $i$.
\end{prop}
\begin{proof}
Lemma \ref{lemma_bseq} shows that every arithmetical structure on $\P_n$ is equal to $\mathbf{A}_n(\mathbf{b})$ for some sequence $\mathbf{b}$ of this type.  Proposition \ref{blowup-path} implies that the arithmetical structures arising from such sequences are distinct.
\end{proof}

In order to complete our first proof of Theorem \ref{main-path-theorem} we need only count the number of sequences described in the statement of Proposition \ref{prop_bseq}.
\begin{lemma}\label{bseq_count}
Let $n$ and $m$ be integers satisfying $2 \le m \le n$. The number of sequences $\mathbf{b} = (b_1,\ldots, b_{n-m})$ satisfying $1\le b_i \le i+m - 2$ and $b_i \le b_{i+1}$ for all $i$ is equal to $B(n-2, n-m)$.
\end{lemma}

Appending an initial string of $m-2$ entries equal to $1$ we see that the sequences of Lemma \ref{bseq_count} are in bijection with sequences $(b_1,\ldots, b_{n-2}) = (1,\ldots, 1, b_{m-1}, \ldots, b_{n-2})$ satisfying $b_i \le i$ and $b_i \le b_{i+1}$. A sequence $(b_1,\ldots, b_k)$ of positive integers satisfying $b_i \le i$ and $b_i \le b_{i+1}$ is called a \emph{ballot sequence of length $k$}.  We need only count the number of ballot sequences of length $n-2$ that begin with at least $m-2$ entries equal to $1$.  The following lemma is equivalent to Lemma \ref{bseq_count}.
\begin{lemma}\label{lemma-initial_string}
The number of sequences $(b_1,\ldots, b_{n-2})$ of positive integers satisfying $b_i \le i$ and $b_i \le b_{i+1}$ that begin with an initial string of at least $m-2$ entries equal to $1$ is given by $B(n-2, n-m)$.
\end{lemma}

\begin{proof}
Recall that $B(n-2, n-2)$ counts the number of lattice paths from $(0,0)$ to $(n-2, n-2)$ that do not cross above the line $y=x$.  It is trivial to note that this also counts the number of lattice paths from $(1,1)$ to $(n-1, n-1)$ that do not cross above the line $y=x$.  Every such lattice path $L$ can be identified uniquely with a ballot sequence $\mathbf{b}_L = (b_1,\ldots, b_{n-2})$ where $(i, b_i)$ is the point along the path on the vertical line $x = i$ with the largest $y$-coordinate.

We want to count the number of these ballot sequences that begin with at least $m-2$ entries equal to $1$, or equivalently, the number of lattice paths from $(1,1)$ to $(n-1,n-1)$ that do not cross above the line $y=x$ and begin with at least $m-2$ east steps.  Reversing the order of such a path, and then replacing each north step with an east step and each east step with a north step, gives a bijection with the set of lattice paths from $(1,1)$ to $(n-1,n-1)$ that do not cross above the line $y=x$ and end with at least $m-2$ north steps.  This set is clearly in bijection with the set of lattice paths from $(1,1)$ to $(n-1, n-m+1)$ that do not cross above the line $y=x$, the number of which is given by $B(n-2,n-m)$.
\end{proof}

\begin{proof}[First proof of Theorem \ref{main-path-theorem}]
Lemma \ref{bseq_count} together with the earlier observation that an arithmetical structure of $\P_n$ has $\rr(1) = m$ if and only if it is equal to $\mathbf{A}_n(\mathbf{b})$ for some sequence $\mathbf{b} = (b_1,\ldots, b_{n-m})$ satisfying $1\le b_i \le i+m - 2$ and $b_i \le b_{i+1}$, completes the first proof of Theorem \ref{main-path-theorem}.
\end{proof}

We give a second proof based on recurrences satisfied by counts for certain classes of lattice paths.
\begin{proof}[Second proof of Theorem \ref{main-path-theorem}]
We argue by induction on $k=\rr(1)$.  In the case $k=1$, we have $B(n-2,n-1)=0$ and Lemma~\ref{path-lemma} implies that there are no arithmetical structures on $\P_n$ with $\rr(1) = 1$, completing the proof in this case.  In the case $k=2$, Theorem \ref{path-count} implies that $\Cballot(n,2) = B(n-2,n-2)=C_{n-2}$.

Now, for fixed $n$, suppose that~\eqref{main-path-thm-eqn} holds for all $j$ satisfying $2\le j \le k < n$.  Suppose that $\rr = (r_1,\ldots, r_n) \in \Arith(\P_n)$ with $\rr(1) = k+1$.  Let $m=\min\{i>1\st r_i=1\}$. Note that $2 \le  m \le n-k+1$.

By Lemma \ref{path-lemma}, the truncated vector $\rr' = (r_1,\ldots, r_m)$ is an arithmetical $r$-structure on $\P_m$ with $\rr'(1) = 2$ and $\rr'' = (r_m,\ldots, r_n)$ is an arithmetical $r$-structure on $\P_{n-m+1}$ with $\rr''(1) = k$. Moreover, every such pair $\rr',\rr''$ gives rise to an arithmetical $r$-structure $\rr$ on $\P_n$ with $\rr(1)=k+1$.  The number of possible choices for such a pair is $\Cballot(m,2) \Cballot(n-m+1,k)$. Therefore, by induction,
\begin{align*}
\Cballot(n,k+1) &= \sum_{m=2}^{n-k+1} \Cballot(m,2) \Cballot(n-m+1,k) \\
&=\sum_{m=2}^{n-k+1} B(m-2,m-2) B(n-m-1,n-m-k+1)\\
&= B(n-2,n-k-1)
\end{align*}
\noindent where the last equality follows from \cite[4.9]{ballot} after a change of variables. (Replace Carlitz's $j,n,k$ with $m-1,n-k-1,n-k$ respectively, and recall that Carlitz's $f(n,k)$ is our $B(n-1,k-1)$.)
\end{proof}

Our next goal is to prove a refined counting result for arithmetical $d$-structures of $\P_n$ with a fixed entry.

\begin{theorem} \label{ballot-path}
For each $i=1,\ldots,n$ and $0\leq k\leq n-2$, the number of arithmetical $d$-structures $(d_1,\ldots, d_n)$ on $\P_n$ with $d_i=n-k-1$ is equal to $B(n-2,k)$.
\end{theorem}

We prove Theorem \ref{ballot-path} in two parts.  We first introduce the notation $d_0=3n-3-\sum_{j=1}^n d_j$ and note that by Corollary~\ref{cor:10} the result extends to the special case $i=0$.  We then define a bijection between triangulations of an $(n+1)$-gon and arithmetical $d$-structures on $\mathcal{P}_n$.  Under this bijection, clockwise rotation of a triangulation induces a correspondence between the set of arithmetical $d$-structures $(d_1,\dots, d_n)$ on $\mathcal{P}_n$ with $d_0 = n-k-1$ (resp. $d_i = n-k-1$) and the set of arithmetical $d$-structures on $\mathcal{P}_n$ with $d_1 = n-k-1$ (resp. $d_{i+1} = n-k-1$).


\begin{proof}
First, by the above definition and by Theorem \ref{critical-group-path} we have $d_0 = \rr(1)-1$.  By Theorem \ref{main-path-theorem} the number of arithmetical $d$-structures with $d_0=n-k-1$ is $B(n-2,k)$.


For the main part of the proof, we fix a labeling of the vertices of an $(n+1)$-gon as $0,1,2,\dots,n$, in clockwise order.  For each triangulation $T$ of the $(n+1)$-gon, define $D(T) = (D_0,D_1,\dots, D_n)$ by
$$D_i = \#\left\{\mathrm{triangles~incident~to~vertex~}i\right\}.$$
The theorem then reduces to the following claim.

\begin{Claim} \label{triangulation-to-d}
The sequence $(d_1,\dots,d_n) = (D_1,\dots, D_n)$ is an arithmetical $d$-structure of $\mathcal{P}_n$.    Moreover, the map $D$ is a bijection from the set of triangulations of the $(n+1)$-gon to the set of arithmetical $d$-structures on $\mathcal{P}_n$.
\end{Claim}

Observe that each triangulation consists of $n-1$ triangles and each triangle will be adjacent to three vertices.  In particular, $\sum_{j=0}^n D_j=3n-3$ so that $D_0 = 3n-3 - \sum_{j=1}^n D_j$.  In particular, $D_0=d_0$, motivating the above notation.


\begin{proof} We now prove Claim~\ref{triangulation-to-d} by induction.  If $n=2$, then the unique and trivial triangulation of a $3$-gon corresponds to the unique arithmetical $d$-structure of $\mathcal{P}_2$, namely $(d_1,d_2) = (1,1)$.  For $n=2$ and this arithmetical $d$-structure, we also have $d_0 = 3-1-1 =1$.

For $n\geq 3$, a triangulation $T$ of an $(n+1)$-gon is obtained by gluing a triangle to the exterior of a triangulation $T'$ of an $n$-gon.  After relabeling the vertices as described below, let $D(T') = (d_0', d_1',\dots, d_{n-1}')$, where $d_j'$ is the number of triangles incident to vertex $j$ in the triangulation $T'$, so by our inductive hypothesis we have that $(d_1',\dots, d_{n-1}')$ is an arithmetical $d$-structure of $\mathcal{P}_{n-1}$.

We next follow a procedure related to Conway and Coxeter's work \cite[(23)]{CC} on frieze patterns and triangulated polygons.
See also \cite[Theorem 2.1]{T} for a more recent description of this procedure.
In this language, the tuple $D(T’)$ is known as the \emph{quiddity sequence} associated to the triangulation $T$’.
For any $0 \le i \le n-1$ we consider the effect of gluing a triangle on the edge between vertices $i$ and $i+1$ of $T'$ and increasing the label on all vertices $j>i$ by one (note that in the case of $i=n-1$ then we are gluing a triangle between vertices $n-1$ and $0$, and we do not need to renumber any vertices). This will create a new triangulation $T$, and if we set $d_j$ to be the number of triangles adjacent to vertex $j$ in this new triangulation, we see that:

\begin{equation}
d_j = \begin{cases}
d'_j \caseif 0 \le j < i-1,\\
d'_{i-1}+1 \caseif j=i-1,\\
1 \caseif j=i,\\
d'_{i}+1 \caseif j=i+1,\\
d'_{j-1} \caseif i+1 < j \le n,
\end{cases}
\end{equation}

agreeing with (\ref{path-subdivision}) and showing that $D(T) = (d_0,d_1,\dots, d_n)$ where $\dd = (d_1,\dots, d_n)$ is the subdivision of $\dd' = (d_1',\dots, d_{n-1}')$ at position $i$.  Hence, by Propositions \ref{blowup-path} and \ref{subdivision-prop}, the resulting $\dd$ is in fact an arithmetical $d$-structure.  We conclude this map is a bijection since the ways in which two sequences of gluing triangles construct the same triangulation is dictated by the same relations  given in Lemma \ref{lemma_bseq}.  This completes the proof of Claim~\ref{triangulation-to-d}.
\end{proof}

By Remark \ref{Ballot-Triangulation}, there are in fact $B(n-2,k)$ triangulations of an $(n+1)$-gon that have $d_i= n-k-1$ triangles incident to a given vertex i.

Alternatively, we observe that if $T$ is a triangulation of an $(n+1)$-gon and $\rho(T)$ is its clockwise rotation, then $D(T) = (d_0,d_1,\dots, d_n)$ and $D(\rho(T)) = (d_n, d_0,d_1,\dots, d_{n-1})$.  Hence rotation induces a bijection between arithmetical $d$-structures $(d_1,\dots, d_n)$ on $\mathcal{P}_n$ with $d_i=n-k-1$ and arithmetical $d$-structures on $\mathcal{P}_n$ with $d_{i+1} = n-k-1$ (for $0\leq i \leq n-1$).  Combining this bijection with the $i=0$ case completes the proof of Theorem~\ref{ballot-path}.
\end{proof}

\begin{Example}
We illustrate the above ideas with the following example, where $T'$ is a triangulation of a pentagon and $T$ is the triangulation of a hexagon obtained by gluing a single triangle to $T'$.

\begin{figure}[h]
\begin{tikzpicture}
[scale=.4,auto=left,minimum size=.7cm,every node/.style={circle,fill=blue!20}]
    \foreach \a/\text in {180/4,120/0,60/1,0/2,240/3}
    \node (\text) at (\a:4cm) {\text};

\foreach \from/\to in {4/0,0/1,1/2,2/3,3/4,1/4,1/3}
    \draw (\from) -- (\to);
    \node [draw=none,fill=none] at (0,-6) {(i) $T'$};

\begin{scope}[shift={(15,0)}] 

\foreach \a/\text in {180/5,120/0,60/1,0/2,300/3,240/4}
    \node (\text) at (\a:4cm) {\text};

\foreach \from/\to in {5/0,0/1,1/2,2/3,3/4,4/5, 2/4,1/4, 5/1}
    \draw (\from) -- (\to);
    \node [draw=none,fill=none] at (0,-6) {(ii) $T$};
\end{scope}
\end{tikzpicture}
\end{figure}

Recalling that for any triangulation $T$ we have $D(T)=(d_0,\ldots d_n)$ where $d_i$ is the number of triangles adjacent to vertex $i$, we see that $D(T') = (1,3,1,2,2)$ and therefore $T'$ corresponds to the arithmetical $d$-structure $(3,1,2,2)$ on $\mathcal{P}_4$.  After gluing on a triangle between vertices $2$ and $3$ and updating the vertex labels, we obtain the triangulation $T$ and the corresponding $D(T)=(1,3,2,1,3,2)$, giving the arithmetical $d$-structure $(3,2,1,3,2)$ on $\mathcal{P}_5$.
\end{Example}

\begin{Remark}
Questions (17) and (18) of \cite{CC} suggest a relationship between the quiddity sequence, which corresponds to the second row of the associated frieze pattern, and the diagonals in the same pattern. In fact, these diagonals correspond to the $\rr$-vectors of the arithmetical structures with a given quiddity sequence as its $\dd$-vector.
\end{Remark}

In the proof of Theorem \ref{ballot-path} we defined a bijection $D$ from triangulations to arithmetical structures on the path, and in Proposition~\ref{prop_bseq} we gave a bijection $\mathbf{A}_n^{-1}$ from arithmetical $d$-structures on $\P_n$ to ballot sequences.  Composing these two bijections together, we get one mapping triangulations to ballot sequences.  Given this composite bijection, it is natural to consider the map $f_n$ from ballot sequences to ballot sequences that corresponds to the action of rotating a given triangulation $T$ of a polygon.  More explicitly, we are interested in the map $f_n$ so that $f_n(\mathbf{A}_n^{-1}(D(T))) = \mathbf{A}_n^{-1}(D(\rho(T)))$. We now give a description of this map.

In particular, let $\mathcal{B}(n)$ denote the set of ballot sequences of length $n$, i.e., $n$-tuples $\mathbf{b} = (b_1,\ldots, b_n)$ consisting of $n$ nonnegative integers, where $b_i \le b_{i+1}$, and $b_i \le i$ for all $i$. We inductively define the bijection $f_n \colon\ \mathcal{B}(n) \rightarrow \mathcal{B}(n)$.  Let $f_1((1)) = (0)$, and $f_1((0)) = (1)$.  Suppose we have defined $f_{n-1}(\mathbf{b})$ for all sequences $\mathbf{b} \in \mathcal{B}(n-1)$.  Given a sequence $\mathbf{b} = (b_1,\ldots, b_{n-1},b_n) \in \mathcal{B}(n)$, let $\mathbf{b}' = (b_1,\ldots, b_{n-1}) \in \mathcal{B}(n-1)$.

Define
\[
f_n(\mathbf{b}) := \begin{cases}
f_{n-1}(\mathbf{b}') \text{ with a } b_n+1 \text{ appended to the end of it} & \text{if } b_n < n,\\
f_{n-1}(\mathbf{b}') \text{ with a } 0 \text{ appended to the beginning of it} & \text{if }b_n=n.
\end{cases}
\]


\begin{Remark}\label{rem:fn}
It is possible to give an explicit description of $f_n(\mathbf{b})$. In particular, let $\mathbf{b} = (b_1,\ldots, b_n) \in\mathcal{B}(n)$ and create a new sequence $\mathbf{b}+(1,1,\ldots,1)$, where addition in the $i$\textsuperscript{th} coordinate is taken modulo $i+1$.
Then $f_n(\mathbf{b})$ is the vector obtained by ``right-justifying" the nonzero entries in $\mathbf{b}+(1,1,\ldots,1)$.  In particular, $f_n(\mathbf{b})$ begins with the same number of zeroes that are in the string $\mathbf{b}+(1,1,\ldots,1)$ and then contains the same numbers as the nonzero entries of this string in the same order.  It is straightforward to verify that this definition of $f_n$ agrees with the one given above.
\end{Remark}

\begin{Example}
If $n=3$, then
\begin{center}
\begin{tabular}{l@{\extracolsep{2cm}}l@{\extracolsep{2cm}}l}
$(1,1,1)\overset{f_3}{\longmapsto} (0,2,2)$, & $(0,1,1)\overset{f_3}{\longmapsto} (1,2,2)$,  & $(0,0,1)\overset{f_3}{\longmapsto} (1,1,2)$, \\
$(1,1,2)\overset{f_3}{\longmapsto} (0,2,3)$, & $(0,1,2)\overset{f_3}{\longmapsto} (1,2,3)$,  & $(0,0,2)\overset{f_3}{\longmapsto} (1,1,3)$, \\
$(1,1,3)\overset{f_3}{\longmapsto} (0,0,2)$, & $(0,1,3)\overset{f_3}{\longmapsto} (0,1,2)$,  & $(0,0,3)\overset{f_3}{\longmapsto} (0,1,1)$, \\
$(1,2,2)\overset{f_3}{\longmapsto} (0,0,3)$, & $(0,2,2)\overset{f_3}{\longmapsto} (0,1,3)$,  & $ (0,0,0) \overset{f_3}{\longmapsto} (1,1,1)$. \\
$(1,2,3)\overset{f_3}{\longmapsto} (0,0,0)$, & $(0,2,3)\overset{f_3}{\longmapsto} (0,0,1)$,  &
\end{tabular}
\end{center}
\end{Example}

We leave the fact that $f_n$ is a bijection and that it corresponds to rotation of a triangulation as exercises for the reader.  This approach leads to an alternative proof to Theorem \ref{ballot-path}.


It is natural to ask for analogues of Theorem~\ref{main-path-theorem} for arithmetical $d$-structures.  That is, how many arithmetical structures $(\mathbf{d}, \mathbf{r})$ on $\P_n$ have $\mathbf{d}(1) = k$?  Results of Aigner and Schulze~\cite{AignerSchulze} give a partial answer.

\begin{prop}
Let $n$ be a positive integer.  The number of arithmetical structures $(\mathbf{d}, \mathbf{r})$ on $\P_{n+2}$ with $\mathbf{r}(1) = 2$ and $\mathbf{d}(1) = k$ is given by
\[
\binom{n-1}{2k-2} 2^{n+1-2k} C_{k-1}.
\]
\end{prop}

\begin{proof}
Call a sequence of integers $a_1, a_2, \dots, a_n$ \emph{admissible} if all $a_i > 1$ and $a_i$ divides $a_{i-1} + a_{i+1}$ for $i = 1,\ldots, n$, where we define $a_0 = a_{n+1} = 1$.  An admissible sequence has a \emph{local maximum} in position $i$ if and only if $a_i = a_{i-1} + a_{i+1}$.  Aigner and Schulze show that the expression in the proposition is the number of admissible sequences of length $n$ with precisely $k$ local maxima \cite[equation~(1)]{AignerSchulze}.

Admissible sequences of length $n$ are in bijection with arithmetical $r$-structures on $\P_{n+2}$ that have $\mathbf{r}(1) = 2$.  Let $(\mathbf{d}, \mathbf{r})$ be an arithmetical structure on $\P_{n+2}$ with $\mathbf{d} = (d_0,d_1,\ldots, d_{n+1})$ and $\mathbf{r} = (r_0, r_1, \ldots, r_{n+1})$.  Then $(r_1,\ldots, r_n)$ is an admissible sequence of length $n$ if and only if $\mathbf{r}(1) = 2$.  For $i$ satisfying $1\le i \le n$ we see that $d_i = 1$ if and only if $r_i = r_{i-1} + r_{i+1}$, that is, when the corresponding admissible sequence of length $n$ has a local maximum in position $i$.  We see that $d_0 = 1$ if and only if $r_0 = r_1 = 1$, and similarly $d_{n+1} = 1$ if and only if $r_{n+1} = r_n = 1$.  Neither of these can happen when $\mathbf{r}(1) = 2$ since $n+1 > 1$.
\end{proof}
It seems significantly more challenging to give formulas for the number of arithmetical structures $(\mathbf{d}, \mathbf{r})$ on $\P_n$ for which $\mathbf{r}(1) = m$ and $\mathbf{d}(1) = k$ when $m > 2$.

When we restrict to arithmetical structures with $\mathbf{r}(1) = 2$ we can derive similar results for other properties considered in \cite{AignerSchulze}.  For example, equation~(2) of that paper gives the number of admissible sequences of length $n$ with precisely $k$ rises $a_i < a_{i+1}$, and equation~(3) gives the number of such sequences without monotone subsequences of length $3$. We do not pursue these directions here.  We do note that equation~(4) of \cite{AignerSchulze} can be interpreted as counting the number of arithmetical structures on $\P_n$ with $\mathbf{r}(1) = 2$ and $d_1 = k$, which is closely related to a special case of Theorem~\ref{ballot-path}.


\section{Cycles}\label{sec:cycles}

As in the case of paths, the arithmetical structures on the $n$-cycle $\C_n$ are controlled by Catalan combinatorics.  The main result of this section, Theorem~\ref{main-cycle-theorem}, states that for each $k\in[n]$, the number of arithmetical structures $(\dd,\rr)$ on $\C_n$ with $\rr(1)=k$ is
\[\multiset{n}{n-k}=\binom{2n-k-1}{n-k}\]
(where the first symbol denotes the number of multisubsets of $[n]$ of cardinality $n-k$), and consequently
\[|\Arith(\C_n)|=\binom{2n-1}{n-1}.\]
As for Theorem \ref{main-path-theorem}, we give two separate proofs of this result.  The first proof constructs an explicit bijection between multisets and arithmetical structures, equivariant with respect to a natural $\Zz_n$-action on each.  The second proof proceeds via enumeration of lattice paths.

In addition, we show that the critical group $K(\C_n,\dd,\rr)$ is always cyclic, with cardinality $\rr(1)=3n-\sum_{j=1}^n d_j$ (Theorem~\ref{critical-group}).

We first state some basic results about arithmetical structures on cycles, some of which have been proved elsewhere.

\begin{prop} \label{c-two}
The cycle $\C_2$ has three arithmetical structures, namely
\[(\dd,\rr)\in\{((2,2),(1,1)),\ ((1,4),(2,1)),\ ((4,1),(1,2))\}.\]
\end{prop}
\begin{proof}
The adjacency matrix of $\C_2$ is $A=\begin{bmatrix}0&2\\2&0\end{bmatrix}$, so in order for $L$ to be singular we must have $d_1d_2=4$, leading to the three possibilities listed above.
\end{proof}

More generally, we note that for the cycle graph $\C_n$, a vector $\rr$ is in the nullspace of $L(\C_n,\dd)$ if and only if $r_{i-1}+r_{i+1} = r_id_i$ for each $i$, where the subscripts are taken mod $n$.  In analogy to Corollary~\ref{path-characterization} for paths, we therefore have:

\begin{prop} \label{cycle-characterization}
Let $\rr=(r_1,\dots,r_n)$ be a primitive positive integer vector.  Then $\rr$ is an arithmetical $r$-structure on $\C_n$ if and only if
\[
r_i|(r_{i-1}+r_{i+1}) \ \ \forall i\in[n]
\]
with the indices taken modulo $n$.
\end{prop}

The next two results are analogous to Propositions~\ref{only-two-path} and~\ref{blowup-path} for paths.

\begin{prop}[{\cite[Thm.~6.5]{arithmetical}}]\label{only-two}
There is exactly one arithmetical structure $(\dd,\rr)$ on $\C_n$ such that $d_i\geq 2$ for all $i$, namely
$\dd=\2=(2,2,\dots,2)$ and $\rr=\1=(1,1,\dots,1)$.
\end{prop}

\begin{prop}[{\cite[pp.~484--485]{Lorenzini89}}; {\cite[Thm.~5.1]{arithmetical}}] \label{subdivision}
\begin{enumerate}
\item Let $n\geq 2$ and let $(\dd',\rr')\in\Arith(\C_n)$.  For $1\leq i\leq n$, define integer vectors $\dd$ and $\rr$ of length $n+1$ as in~\eqref{path-subdivision}, with the conventions $d'_0=d'_n$ and $r'_0=r'_n$.
Then $(\dd,\rr)$ is an arithmetical structure on~$\C_{n+1}$.
Moreover, the cokernels of $L(\C_{n+1},\dd)$ and $L(\C_n,\dd')$ are isomorphic.
\item Let $n\geq 3$ and let $(\dd,\rr)\in\Arith(\C_n)$ such that $d_{i-1}>d_i=1<d_{i+1}$ for some $i\in[n]$.  Define integer vectors $\dd',\rr'$ of length~$n-1$ as in~\eqref{path-smoothing}, with the conventions $d_0=d_n$ and $r_0=r_n$.  Then $(\dd',\rr')$ is an arithmetical structure on $\C_{n-1}$, and the cokernels of $L(\C_n,\dd)$ and $L(\C_{n-1},\dd')$ are isomorphic.
\end{enumerate}
\end{prop}

As in Proposition~\ref{blowup-path}, we say that $(\dd,\rr)$ is the \emph{subdivision of $(\dd',\rr')$ at position~$i$} and $(\dd',\rr')$ is the \emph{smoothing} of $(\dd,\rr)$ at position~$i$.
Proposition~\ref{subdivision} can be proved in the same way as Proposition~\ref{blowup-path}, because subdivision and smoothing are local operations that only concern vertices of degree 2.

Recall that $\rr(1)$ denotes the number of occurrences of 1 in an arithmetical $r$-structure $\rr$.
\begin{cor}
We have $\rr(1)>0$ for all arithmetical $r$-structures on $\C_n$.
\end{cor}
\begin{proof}
For $n=2$, the assertion follows immediately from Proposition~\ref{c-two}.  For $n\ge3$, the claim is obvious if $\dd=\2$ and $\rr=\1$. If $\dd\neq\2$, then by Proposition \ref{only-two} and Lemma~\ref{isolated-ones-in-d} there exists $i\in [n]$ such that $d_{i-1}>d_i=1<d_{i+1}$. But then $(\dd,\rr)$ is the subdivision of an arithmetical structure $(\dd',\rr')$ on $\C_{n-1}$ as described in Proposition~\ref{subdivision}. Since $\rr(1)=\rr'(1)$, the result follows by induction.
\end{proof}

\begin{theorem}\label{critical-group}
Let $(\dd,\rr)\in\Arith(\C_n)$ be an arithmetical structure of the cycle.  Then
\begin{equation} \label{r-and-d}
\rr(1)=3n-\sum_{j=1}^n d_j
\end{equation}
and
\begin{equation} \label{critical-cyclic}
K(\C_n,\dd,\rr)=\mathbb{Z}_{\rr(1)}.
\end{equation}
\end{theorem}
\begin{proof}
We induct on $n$.  First, in the base case $n=2$, both claims can be checked by direct computation for the three arithmetical structures listed in Proposition~\ref{c-two}.

Second, if $\dd=\2$ and $\rr=\1$, then $3n-\sum_{i=1}^n d_i=n=\rr(1)$.  In this case $L$ is the Laplacian, and $K(\C_n,\2,\1)=\mathbb{Z}_n$, the standard critical group of $\C_n$.

Third, suppose that $n\geq3$ and $\dd\neq\2$.  Then by Proposition~\ref{subdivision},
$(\dd,\rr)$ can be obtained from subdividing some $(\dd',\rr')\in\Arith(\C_{n-1})$ at position~$i$, and $K(\C_n,\dd,\rr)=K(\C_{n-1},\dd',\rr')$.  The recursive description of $\rr$ in \eqref{path-smoothing} implies that $\rr(1)=\rr'(1)$, establishing the isomorphism \eqref{critical-cyclic}.  Moreover,
\begin{align*}
\rr(1) &= \rr'(1) = 3(n-1) - \sum_{j=1}^{n-1} d'_j && \text{(by induction)}\\
&= 3(n-1) - \left(\sum_{j=1}^n d_j - (2+d_i)\right)\\
&= 3n - \sum_{j=1}^n d_j && \text{(because $d_i=1$).}
\end{align*}
\end{proof}

We now come to the main theorem of this section, enumerating arithmetical $r$-structures on $\C_n$ by the value of $\rr(1)$.  We first need to set up notation for multisets.  A \emph{multiset} is a list $S=[a_1,\dots,a_\ell]$, where order does not matter, and repeats are allowed.  The number $\ell$ is the \emph{size} or \emph{cardinality} of $S$.  We will use square brackets to distinguish multisets from ordinary sets.  
For a set $T$, if $a_i\in T$ for all $i$ then we say that $S$ is a \emph{multisubset} of $T$.
We write $\Multiset_{\ell}(T)$ to denote the set of multisubsets of $T$ of size $\ell$
and let $\multiset{n}{\ell} = |\Multiset_{\ell}([n])| = \binom{n+\ell-1}{\ell}$.
Likewise, $\Multiset_{\leq\ell}(T)$ denotes the set of multisubsets of $T$ of size at most $\ell$.
If $T=[n]$, we abbreviate $\Multiset_{\ell}([n])$ as $\Multiset_{\ell}(n)$ and $\Multiset_{\leq\ell}([n])$ as $\Multiset_{\leq\ell}(n)$. There is a bijection $\Multiset_{n-1}(n+1)\to\Multiset_{\leq n-1}(n)$ that erases all instances of $n+1$, which implies that $\sum_{\ell=0}^{n-1} \multiset{n}{\ell} = \multiset{n+1}{n-1}$.

\begin{theorem} \label{main-cycle-theorem}
Let $1\leq k\leq n$ and $\ell=n-k$.  Then
\[\#\{(\dd,\rr)\in\Arith(\C_n) \st \rr(1)=k\} = \multiset{n}{n-k}=\binom{2n-k-1}{n-k}.\]
In particular, the total number of arithmetical structures on $\C_n$ is
\[\sum_{k=1}^n \multiset{n}{n-k} = \sum_{\ell=0}^{n-1} \multiset{n}{\ell}
 = \multiset{n+1}{n-1} = \binom{2n-1}{n-1}.\]
\end{theorem}

Combining Theorems~\ref{critical-group} and~\ref{main-cycle-theorem} immediately yields the following result.
\begin{cor}
For $n\geq 2$ and $2n\leq k\leq 3n-1$, we have
\[\#\left\{(\dd,\rr)\in\Arith(\C_n) \st k = \sum_{j=1}^n d_j\right\} = \multiset{n}{k-2n}=\binom{k-n-1}{k-2n}.\]
\end{cor}

As we did with Theorem \ref{main-path-theorem}, we give two separate combinatorial proofs of Theorem~\ref{main-cycle-theorem}.  The first is bijective and the second involves recurrences for lattice paths.  Before giving the first proof, we describe actions $\rho$ and $\phi$ of the cyclic group $\Z_n$ (written additively) on the sets $\Arith(\C_n)$ and $\Multiset_{\ell}(n)$.  Specifically, $c\in \Z_n$ acts on $\Arith(\C_n)$ by rotating positions:
\[\rho_c(r_1,\dots,r_n) = (r_{c+1},\dots,r_n,r_1,\dots,r_c),\]
and on multisets by rotating values:
\[ \phi_c([a_1,\dots,a_\ell]) = [a_1+c,\dots,a_\ell+c], \]
with all elements taken modulo~$n$.  For example, the $\rho$-orbit of the arithmetical $r$-structure $(3,1,2,1,2)$ on $\C_5$ is
\[\big\{(3,1,2,1,2), \ \  (1,2,1,2,3), \ \  (2,1,2,3,1), \ \  (1,2,3,1,2), \ \  (2,3,1,2,1)\big\} \]
and the $\phi$-orbit of $[1,3,3,4]$ in $\Multiset_4(5)$ is
\[\big\{[1,3,3,4], \ \ [2,4,4,5], \ \ [3,5,5,1], \ \ [4,1,1,2], \ \ [5,2,2,3]\big\}.\]

\begin{proof}[First proof of Theorem~\ref{main-cycle-theorem}]
We will construct an explicit bijection
\[\Omega:\Multiset_{\leq n-1}(n)\to\Arith(\C_n),\]
which for each multiset constructs an arithmetical structure on the cycle.  The method of doing so is analogous to the first proof of Theorem \ref{main-path-theorem}, in which we constructed an arithmetical structure on the path for each ballot sequence. We note that our bijection $\Omega$ will have the following properties:
\begin{enumerate}
\item $\Omega(S)=(\dd,\rr)$ with $\rr(1)=n-|S|$ for all $S$.
\item $\Omega$ is equivariant with respect to the actions of $\Z_n$ just described, i.e., $\Omega(\phi_t(S))=\rho_t(\Omega(S))$.
\item Given a nonempty multiset $S$, let $\tilde{S} = \phi_c(S)$ be the element of the $\phi$-orbit of $S$ that is first in reverse-lex\footnote{If $A=[a_1\leq\cdots\leq a_k]$ and $B=[b_1\leq\cdots\leq b_k]$ are $k$-multisets, then $A$ precedes $B$ in reverse-lex order if and only if $a_j<b_j$, where $j=\max\{i\st a_i\neq b_i\}$.} order.  Then $\Omega(\tilde{S}) = \tilde{\mathbf{r}}$ is first in reverse-lex order in its $\rho$-orbit.  In particular, $\tilde{r}_n = 1$.
\end{enumerate}

We begin by setting $\Omega(\emptyset) = (\2,\1)$. Given a nonempty multiset $S$, let $\tilde{S}$ be defined as in property (C) above.  Write $\tilde{S}$ as $[s_1,\ldots,s_\ell]$ where each $s_i \le s_{i+1}$.  Because $|S|=\ell<n$, some element of its $\phi$-orbit, hence $\tilde S$, contains no instance of $n$. This implies $s_{\ell} < n$.  Also, $s_1 = 1$, since otherwise subtracting $1$ from each element produces an element of the $\phi$-orbit earlier in reverse-lex order.

We now define an arithmetical $r$-structure $\tilde{\rr}=(\tilde{r}_1,\dots,\tilde{r}_{n})$ (which we regard as labels on the vertices of $\C_n$) by the following algorithm, with the steps numbered for later reference. Much as Proposition \ref{subdivision-prop} constructed arithmetical structures on paths by describing a series of subdivisions on the normal Laplacian arithmetical structure on a shorter path, this algorithm will start with the Laplacian arithmetical structure on $\C_1$ and make a sequence of subdivisions based on the given multiset.

{\bf Algorithm A}
\begin{enumerate}
\item[Input:] A multiset $\tilde{S}=\phi_c(S)=[s_1,\ldots,s_\ell]$ as above.\\
\item[(1)] Initialize $\tilde{r}_0=1$ and $n_0=1$ on $\C_1$ (by convention, we set $\tilde{r}_0=\tilde{r}_n$).
\item[(2)] For each integer $i$ with $1 \le i \le \ell$, we construct an $r$-structure on the cycle graph $\C_{n_i}$ ($1=n_0< n_1 \leq n_2 \leq n_3 \leq \dots \leq n_\ell \le n$) as follows:
\begin{enumerate}
\item If $n_{i-1}<s_i$ add vertices with a label of $\tilde{r}_j=1$ until there are $s_i$ vertices.  Then add a vertex with label $\tilde{r}_{s_i}=2$, and set $n_i=s_i+1$
\item If $n_{i-1}=s_i$, add a vertex with label $\tilde{r}_{s_i}=\tilde{r}_{s_i-1}+1$, and set $n_i=s_i+1$
\item If $n_{i-1} > s_i$, then insert a vertex with label $\tilde{r}_{s_i}+\tilde{r}_{s_i-1}$ into position $s_i$, which will have the effect of moving all later labels forward one vertex. Set $n_i=n_{i-1}+1$
\end{enumerate}
\item[(3)] If $n_{\ell}<n$, add $n-n_\ell$ vertices with a label of $1$.
\item[(4)] The resulting arithmetical $r$-structure is $(\tilde{r}_1,\dots, \tilde{r}_n)=\tilde{\rr}=\Omega(\tilde{S})$, recalling that $\tilde{r}_0=\tilde{r}_n$.  Set $\rr=\Omega(S) = \rho_{-c}(\tilde{\rr})$.
\end{enumerate}
We emphasize that exactly one of steps (2a), (2b) and (2c) is executed for each~$i$. Moreover, we claim that after each iteration of step~(2) the labeled vertices form an arithmetical $r$-structure on a smaller cycle $\C_{n_i}$, and in particular the output $\rr$ produced by this algorithm is indeed an arithmetical $r$-structure on $\C_n$ for the following reasons:
\begin{enumerate}
\item[i)] Because we know that $s_1=1$, after the first iteration of step~(2) we have the vector $(1,2)$ which is an arithmetical $r$-structure on $\C_2$ by Proposition \ref{c-two}.
\item[ii)] After each iteration of the procedure in step~(2a), an arithmetical $r$-structure on $\C_{n_i}$ with $\tilde{r}_0=1$ now also ends with a sequence of $1$'s followed by a $2$.  The divisibilities of Proposition \ref{cycle-characterization} are thus preserved.
\item[iii)] The procedures in steps~(2b) and (2c) amount to inserting a vertex that is labeled with the sum of the labels of its neighbors, which is essentially the subdivision operation discussed for paths.  In particular, the divisibilities of Proposition \ref{cycle-characterization} still hold.
\item[iv)] Step~(3)  will simply add a string of vertices labeled with a $1$ where there was only one such vertex before.  Therefore, this also preserves the fact that the output is an arithmetical $r$-structure.
\item[v)] Finally, applying a cyclic rotation as in step~(4) does not break this property.
\end{enumerate}

Before analyzing this algorithm further, we give two examples, with $n=6$.  Index the vertices of $\C_6$ according to the following figure:
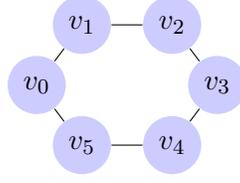
\begin{center}
\begin{tikzpicture}
 [scale=.4,auto=left,minimum size=.7cm,every node/.style={circle,fill=blue!20}]
  \node (a) at (-1.5,2) {$v_1$};
  \node (b) at (1.5,2) {$v_2$};
  \node (c) at (3,0) {$v_3$};
  \node (d) at (1.5,-2) {$v_4$};
  \node (e) at (-1.5,-2) {$v_5$};
  \node (f) at (-3,0) {$v_0$};
\foreach \from/\to in {a/b,b/c,c/d,d/e,e/f,f/a}
    \draw (\from) -- (\to);
\end{tikzpicture}
\end{center}
For $\tilde{S} = [1,1,3,5]$, Algorithm~A proceeds as in Figure~\ref{AlgmAEx:1}.  Each figure shows one iteration of the procedure in step~2.

\begin{figure}[h]
\begin{tikzpicture}
 [scale=.35,auto=left,minimum size=.7cm,every node/.style={circle,fill=blue!20}]
  \coordinate (a) at (-1.5,2) {};
  \coordinate (b) at (1.5,2) {};
  \coordinate (c) at (3,0) {};
  \coordinate (d) at (1.5,-2) {};
  \coordinate (e) at (-1.5,-2) {};
  \node (f) at (-3,0) {1};
\foreach \from/\to in {a/b,b/c,c/d,d/e,e/f,f/a}
    \draw (\from) -- (\to);
  \node [draw=none,fill=none] at (0,-4) {(i)};
\end{tikzpicture}
\hfill
\begin{tikzpicture}
 [scale=.35,auto=left,minimum size=.7cm,every node/.style={circle,fill=blue!20}]
  \node (a) at (-1.5,2) {2};
  \coordinate (b) at (1.5,2) {};
  \coordinate (c) at (3,0) {};
  \coordinate (d) at (1.5,-2) {};
  \coordinate (e) at (-1.5,-2) {};
  \node (f) at (-3,0) {1};
\foreach \from/\to in {a/b,b/c,c/d,d/e,e/f,f/a}
    \draw (\from) -- (\to);
  \node [draw=none,fill=none] at (0,-4) {(ii)};
\end{tikzpicture}
\hfill
\begin{tikzpicture}
  [scale=.35,auto=left,minimum size=.7cm,every node/.style={circle,fill=blue!20}]
  \node (a) at (-1.5,2) {3};
  \node (b) at (1.5,2) {2};
  \coordinate (c) at (3,0) {};
  \coordinate (d) at (1.5,-2) {};
  \coordinate (e) at (-1.5,-2) {};
  \node (f) at (-3,0) {1};
\foreach \from/\to in {a/b,b/c,c/d,d/e,e/f,f/a}
    \draw (\from) -- (\to);
  \node [draw=none,fill=none] at (0,-4) {(iii)};
\end{tikzpicture}
\hfill
\begin{tikzpicture}
  [scale=.35,auto=left,minimum size=.7cm,every node/.style={circle,fill=blue!20}]
  \node (a) at (-1.5,2) {3};
  \node (b) at (1.5,2) {2};
  \node (c) at (3,0) {3};
  \coordinate (d) at (1.5,-2) {};
  \coordinate (e) at (-1.5,-2) {};
  \node (f) at (-3,0) {1};
\foreach \from/\to in {a/b,b/c,c/d,d/e,e/f,f/a}
    \draw (\from) -- (\to);
  \node [draw=none,fill=none] at (0,-4) {(iv)};
\end{tikzpicture}
\hfill
\begin{tikzpicture}
  [scale=.35,auto=left,minimum size=.7cm,every node/.style={circle,fill=blue!20}]
  \node (a) at (-1.5,2) {3};
  \node (b) at (1.5,2) {2};
  \node (c) at (3,0) {3};
  \node (d) at (1.5,-2) {1};
  \node (e) at (-1.5,-2) {2};
  \node (f) at (-3,0) {1};
\foreach \from/\to in {a/b,b/c,c/d,d/e,e/f,f/a}
    \draw (\from) -- (\to);
  \node [draw=none,fill=none] at (0,-4) {(v)};
\end{tikzpicture}
\caption{Algorithm A, with $n=6$ and $\tilde{S}=[1,1,3,5]$.\label{AlgmAEx:1}}
\end{figure}
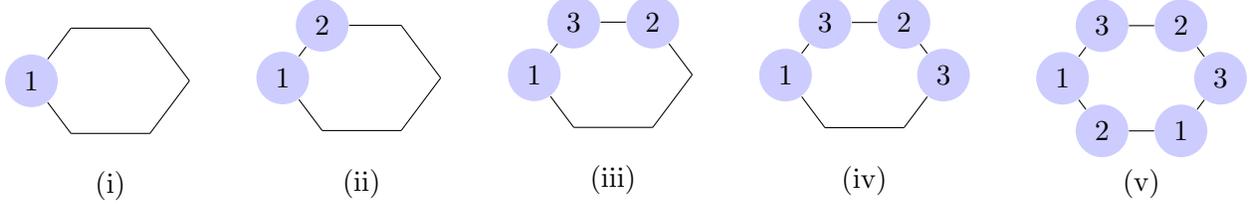

\begin{enumerate}[label=(\roman*)]
\item The leftmost figure is the result of the initialization (step~1).
\item When $i=1$, we have $s_i=1=n_{i-1}$. Step~(2b) inserts a vertex with label $1+1=2$ in position 1 and sets $n_1=2$.  Note that we always have $s_1=1=n_0$, so this will always be the output of the first iteration.
\item  When $i=2$ we have $s_i=1<n_{i-1}$. Step~(2c) inserts a vertex with label $1+2=3$ in position 1 and moves the vertex labeled 2 into position 2.  We then set $n_2=3$.
\item  When $i=3$ we have $s_i=3=n_2$. Step~(2b) places a vertex with label $r_{s_2}+1=3$ in position 3, and sets $n_3=4$.
\item  When $i=4$ we have $s_i=5>n_{i-1}$. Step~(2a) places a vertex with a label of $1$ in position 4, and a vertex with label $2$ in position 5. We now set $n_4=6=n$, so step (3) does not occur.
\end{enumerate}

As a second example, if $\tilde{S} = [1,1,4,4]$ then the algorithm proceeds as in Figure~\ref{AlgmAEx:2}.

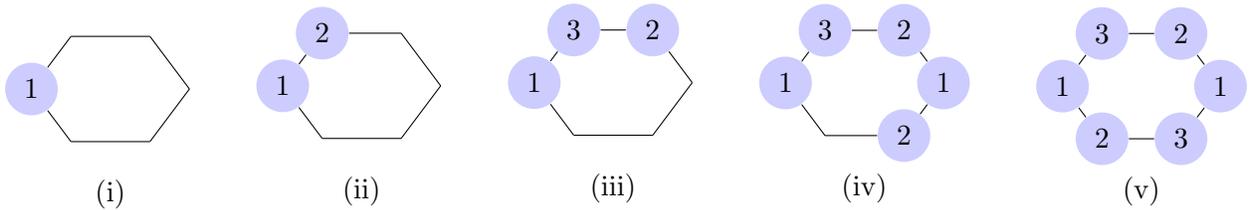
\begin{figure}[h]
\begin{tikzpicture}
 [scale=.35,auto=left,minimum size=.7cm,every node/.style={circle,fill=blue!20}]
  \coordinate (a) at (-1.5,2) {};
  \coordinate (b) at (1.5,2) {};
  \coordinate (c) at (3,0) {};
  \coordinate (d) at (1.5,-2) {};
  \coordinate (e) at (-1.5,-2) {};
  \node (f) at (-3,0) {1};
\foreach \from/\to in {a/b,b/c,c/d,d/e,e/f,f/a}
    \draw (\from) -- (\to);
  \node [draw=none,fill=none] at (0,-4) {(i)};
\end{tikzpicture}
\hfill
 \begin{tikzpicture}
 [scale=.35,auto=left,minimum size=.7cm,every node/.style={circle,fill=blue!20}]
  \node (a) at (-1.5,2) {2};
  \coordinate (b) at (1.5,2) {};
  \coordinate (c) at (3,0) {};
  \coordinate (d) at (1.5,-2) {};
  \coordinate (e) at (-1.5,-2) {};
  \node (f) at (-3,0) {1};
\foreach \from/\to in {a/b,b/c,c/d,d/e,e/f,f/a}
    \draw (\from) -- (\to);
  \node [draw=none,fill=none] at (0,-4) {(ii)};
\end{tikzpicture}
\hfill
\begin{tikzpicture}
  [scale=.35,auto=left,minimum size=.7cm,every node/.style={circle,fill=blue!20}]
  \node (a) at (-1.5,2) {3};
  \node (b) at (1.5,2) {2};
  \coordinate (c) at (3,0) {};
  \coordinate (d) at (1.5,-2) {};
  \coordinate (e) at (-1.5,-2) {};
  \node (f) at (-3,0) {1};
\foreach \from/\to in {a/b,b/c,c/d,d/e,e/f,f/a}
    \draw (\from) -- (\to);
  \node [draw=none,fill=none] at (0,-4) {(iii)};
\end{tikzpicture}
\hfill
\begin{tikzpicture}
  [scale=.35,auto=left,minimum size=.7cm,every node/.style={circle,fill=blue!20}]
  \node (a) at (-1.5,2) {3};
  \node (b) at (1.5,2) {2};
  \node (c) at (3,0) {1};
  \node (d) at (1.5,-2) {2};
  \coordinate (e) at (-1.5,-2) {};
  \node (f) at (-3,0) {1};
\foreach \from/\to in {a/b,b/c,c/d,d/e,e/f,f/a}
    \draw (\from) -- (\to);
  \node [draw=none,fill=none] at (0,-4) {(iv)};
\end{tikzpicture}
\hfill
\begin{tikzpicture}
 [scale=.35,auto=left,minimum size=.7cm,every node/.style={circle,fill=blue!20}]
  \node (a) at (-1.5,2) {3};
  \node (b) at (1.5,2) {2};
  \node (c) at (3,0) {1};
  \node (d) at (1.5,-2) {3};
  \node (e) at (-1.5,-2) {2};
  \node (f) at (-3,0) {1};
\foreach \from/\to in {a/b,b/c,c/d,d/e,e/f,f/a}
    \draw (\from) -- (\to);
  \node [draw=none,fill=none] at (0,-4) {(v)};
\end{tikzpicture}
\caption{Algorithm A, with $n=6$ and $\tilde{S}=[1,1,4,4]$.\label{AlgmAEx:2}}
\end{figure}

We now show that the function $\Omega:\Multiset_{\leq n-1}(n)\to\Arith(\C_n)$ defined by Algorithm~A satisfies the desired properties (A), (B), and (C), and is a bijection. First, each iteration of the procedure in step~(2) adds one new vertex with a label $\tilde{r}_i$ greater than $1$.  Thus, the number of $1$'s in $\Omega(S)$, which is unaffected by the cyclic rotation $\rr = \rho_{-c}(\tilde{\rr})$ of Step (4), equals $\Omega(S)(1)=n-|S|$, and we can regard $\Omega$ as the union of maps
\[
\Omega_\ell:\Multiset_{\ell}(n)\to\{\text{arithmetical } r\text{-structures on } \C_n \text{ with }  \rr(1)=n-\ell\}.
\]
Second, to see that this map is equivariant, we let $T=\phi_t(S)$ for $t \in \Z_n$.  Then one can check that $\tilde{T}=\tilde{S}=\phi_{-t+c}(T)$.  We then get that $\Omega(T) = \rho_{t-c}(\tilde{\rr})= \rho_t(\Omega(S))$, as desired.

Third, by starting with $\tilde{S} = [s_1,s_2,\dots, s_\ell]$, which is the multiset in its $\phi$-orbit that comes first in reverse-lex order, the arithmetical $r$-structure $\tilde{\rr} = \Omega(\tilde{S})$ is constructed so that $\tilde{r}_j = 1$ for $s_\ell< j \leq n$ and $\tilde{r}_{s_\ell} > 1$.  The only way that $\tilde{\rr}$ can contain a longer string of $1$'s earlier is if there exists a choice of $(i,i+1)$ so that the gap $s_{i+1}-s_i-1$ is greater than the gap $n-s_\ell$.  However that would contradict that fact that $\tilde{S}$ is first in reverse-lex order since otherwise adding $n-s_{i+1}+1$ (modulo $n$) to all entries of $\tilde{S}$ produces a new maximal entry $n - s_{i+1}+s_i+1$ which would be smaller than $s_\ell$.

We will now show that each $\Omega_\ell$ is a bijection.  This is clear for $\ell=0$, i.e., when $S=\emptyset$.  When $\ell=1$, we have $S=[a_1]$ and $\tilde S=[1]=\phi_{n-a_1+1}(S)$.  The procedure in step~(2) is executed only once, and step~(3) sets $\tilde{\rr}=(2,1,\dots,1)$, so the output of the algorithm via step~(4) will be $\rho_{-n+a_1-1}(2,1,\dots,1) = (r_1,\ldots,r_n)$ where $r_{n-a_1+2}=2$ and all other $r_i=1$.  Again, we see that $\Omega_1$ is a bijection.

Suppose now that $\ell\geq 2$.  After each iteration of the procedure in step~(2), the label in position $s_i$ is a \emph{local maximum}; that is $\tilde{r}_{s_i-1}<\tilde{r}_{s_i}$, and either $\tilde{r}_{s_i+1}<\tilde{r}_{s_i}$ or there is not a vertex $s_i+1$, in which case $\tilde{r}_0=\tilde{r}_n=1$ would be the next label on a vertex.  We claim in addition that if $m>{s_i}$ then $r_m$ is not a local maximum.  This is clear if step~(2a) or (2b) was executed, for then $s_i=n_i-1$.  On the other hand, if step~(2c) was just executed $b$ times in a row, then each step inserted a label that is greater than the label to its right, and each insertion occurred to the right of all previous insertions (since the sequence $(s_i)$ is in weakly increasing order), which proves the claim.

Therefore, we can recover $\tilde S$ from $\rr$ by the following algorithm.  First, let $\tilde\rr=\rho_{c}(\rr)$ be the element of the $\rho$-orbit of $\rr$ that is first in reverse-lex order (so in particular $\tilde r_n=1$).  Label the vertices of $\C_n$ with $\tilde\rr$, and perform the following steps.

{\bf Algorithm B}
\begin{enumerate}
\item[Input:] An arithmetical $r$-structure $\tilde{\rr}=\rho_c(\rr)=(\tilde{r}_1,\ldots,\tilde{r}_n)$ as above.\\
\item[(1)] Let $\tilde S$ be the empty multiset.
\item[(2)] Let $j$ be the greatest integer $1 \le j \le n$ so that $\tilde{r}_j$ is a local maximum, and add $j$ to the multiset $\tilde S$.
\item[(3)] Delete $\tilde{r}_j$ from $\tilde\rr$.  What remains is an arithmetical $r$-structure on the graph $\C_{n-1}$.
\item[(4)] Repeat the previous two steps until we are left with the arithmetical $r$-structure $\1$ on the graph $\C_{n-\ell}$.  The multiset $\tilde S$ will now contain $\ell = n-\rr(1)$ elements, and will be first in reverse-lex order in its $\phi$-orbit.
\end{enumerate}

Having recovered $\tilde S$, we set $S=\phi_{-c}(\tilde S)$.

Steps~(2) and~(3) of Algorithm B will be executed exactly $\ell$ times, since each iteration removes one entry greater than 1 from $\tilde{\rr}$.   The fact that we have an arithmetical $r$-structure on $\C_{n-1}$ after step~(3) follows from part~(B) of Proposition~\ref{subdivision}.
\end{proof}

We now give a second proof of Theorem~\ref{main-cycle-theorem} using enumeration of lattice paths. This proof relies on our refined count for arithmetical structures on paths given in Theorem~\ref{main-path-theorem}.
Before giving the proof we explain how an arithmetical $r$-structure on a cycle gives rise to arithmetical $r$-structures on paths of vertices contained within that cycle.

\begin{lemma}\label{cycle-lemma}
Let $n \ge 2$.
\begin{enumerate}
\item Suppose that $\rr = (r_1,\ldots, r_n)$ is an arithmetical $r$-structure on $\C_n$ with $r_j = 1$ for some $1 \le j \le n$.  Then $(r_j, r_{j+1},\ldots, r_n, r_1,\ldots, r_j)$ is an arithmetical $r$-structure on $\P_{n+1}$.

\item Suppose that $\rr = (r_1,\ldots, r_n)$ is an arithmetical $r$-structure on $\C_n$ with $r_\alpha = 1$ and $r_\beta = 1$ for some $1 \le \alpha < \beta \le n$.  Then $(r_\alpha, r_{\alpha+1},\ldots, r_\beta)$ is an arithmetical $r$-structure on $\P_{\beta-(\alpha-1)}$ and $(r_\beta,r_{\beta+1},\ldots, r_n, r_1,\ldots, r_\alpha)$ is an arithmetical structure on $\P_{n-(\beta-\alpha)+1}$.
\end{enumerate}
\end{lemma}

\begin{proof}
The proofs are similar to those in Lemma~\ref{path-lemma}. For (A), note that $\rr$ is an arithmetical $r$-structure on $\C_n$ if and only if there exists a positive integral vector $\dd$ such that the following equations hold, where the indices are taken modulo~$n$:
\[
r_id_i=r_{i-1}+r_{i+1} \qquad 1\le i\le n.
\]
Without loss of generality, assume that $r_1=1$. Then we have the following set of equations, showing that $(r_1,r_2,\dots,r_n,r_1)$ is an arithmetical $r$-structure on $\P_{n+1}$:
\begin{align*}
\tilde{d}_1 &:= r_2\\
d_ir_i &= r_{i-1}+r_{i+1} \qquad 1<i\le n\\
\tilde{d}_{n+1} &:= r_n.
\end{align*}
Similarly for (B), if $r_1=r_\beta=1$, then we have the equations:
\begin{align*}
\tilde{d}_1 &:= r_2\\
d_ir_i &= r_{i-1}+r_{i+1} \qquad 1<i<\beta\\
\tilde{d}_{\beta} &:= r_{\beta-1},
\end{align*}
showing that $(r_1,r_2,\dots,r_\beta)$ is an arithmetical $r$-structure on $\P_{\beta}$. An identical argument shows that $(r_\beta,r_{\beta+1},\dots,r_n,r_1)$ is an arithmetical $r$-structure on $\P_{n-\beta+2}$.
\end{proof}

\begin{proof}[Second Proof of Theorem~\ref{main-cycle-theorem}]
First, consider the case $k=1$.  By Lemma \ref{cycle-lemma}, for each $j\in[n]$, the map
\[
(r_1,\dots,r_n) \mapsto (r_j,r_{j+1},\dots,r_n,r_1,\dots,r_j)
\]
is a bijection between arithmetical $r$-structures $(r_1,\dots,r_n)$ on $\C_n$ with a unique 1 in position $j$, and arithmetical $r$-structures on $\P_{n+1}$ with exactly two entries equal to $1$.  By Theorem~\ref{main-path-theorem},
\begin{align*}
\#\{\rr\in\Arith(\C_n) \st \rr(1)=1\} &= n\cdot\#\{\rr\in\Arith(\P_{n+1}) \st \rr(1)=2\}\\
&= n\cdot \Cballot(n+1,2)\\
&= n\cdot C_{n-1} = \binom{2n-2}{n-2} = \multiset{n}{n-1}
\end{align*}
as claimed.

Now suppose that $k > 1$. There are $\binom{2n-k-1}{n-k}$  lattice paths from $(0,0)$ to $(n-1,n-k)$ consisting of north and east steps.  Since $n-1 > n-k$, every such lattice path $P$ touches the line $x=y$ for a last time at some point $(z,z)$, where $0 \le z \le n-k$.  That is, $P$ consists of a lattice path $P_1$ from $(0,0)$ to $(z,z)$, followed by a step east to $(z+1,z)$, followed by a lattice path $P_2$ from $(z+1,z)$ to $(n-1,n-k)$ that does not cross (although it may touch) the diagonal line $x = y+1$.  There are $\binom{2z}{z} = (z+1) C_z$ choices for the subpath $P_1$.  The possibilities for the path $P_2$ are in bijection with lattice paths from $(0,0)$ to $(n-z-2,n-z-k)$ that do not cross above the line $x=y$. This gives $B(n-z-2,n-z-k)=\Cballot(n-z,k)$ possible paths $P$.  Therefore
\begin{equation} \label{lattice-path-count}
\binom{2n-k-1}{n-k} = \sum_{z=0}^{n-k} (z+1) C_z \cdot \Cballot(n-z,k).
\end{equation}
We show that this expression counts the arithmetical $r$-structures on $\C_n$ with $\rr(1) = k$.

It is now convenient to think of the vertices of $\C_n$ as $v_0,\ldots, v_{n-1}$.  Let $\rr=(r_0,\dots,r_{n-1})$ be an arithmetical $r$-structure on $\C_n$ such that $\rr(1)=k$.  Let
\[
\alpha = \min\{i \st r_i=1\}, \qquad \beta=\max\{i\st r_i=1\}.
\]
Note that $0 \le \alpha < \beta \le n-1$.

First, by part (B) of Lemma \ref{cycle-lemma}, $\rr'=(r_\alpha,r_{\alpha+1},\dots,r_{\beta -1},r_\beta)$ is an arithmetical $r$-structure on the path $\P_{\beta-\alpha+1}$ with $\rr'(1)=k$.  In particular, $\beta-\alpha + 1 \ge k$.  By Theorem~\ref{main-path-theorem}, the number of possibilities for $\rr'$ is $\Cballot(\beta-\alpha+1,k)$.

Second, observe that $\rr''=(r_\beta, r_{\beta+1},\dots, r_{n-1}, r_0,r_1, \dots, r_\alpha)$ is an arithmetical $r$-structure on the path $\P_{n-(\beta-\alpha)+1}$ with $\rr''(1)=2$.  Again by Theorem~\ref{main-path-theorem}, the number of possibilities for $\rr''$ is $\Cballot(n-\beta+\alpha+1,2)$, which is equal to the Catalan number $C_{n-(\beta-\alpha)-1}$.

Let $z = n-(\beta-\alpha)-1$ and note that $0\le z \le n-k$.  For fixed $n, k$, each choice of $\alpha$ and $\beta$ satisfying $\beta-\alpha + 1 \ge k$ gives exactly $C_z \cdot \Cballot(n-z,k)$ possible arithmetical $r$-structures.  Moreover, each value of $z$ arises from precisely $z+1$ pairs $(\alpha,\beta)$, namely $(0,n-z-1),(1,n-z),\dots,(z,n-1)$.  Therefore, the number of possible arithmetical structures is
\[
\sum_{z=0}^{n-k} (z+1) C_z \cdot \Cballot(n-z,k)
\]
which, combined with \eqref{lattice-path-count}, completes the proof.
\end{proof}

We do not know as much about the distribution of single digits in the arithmetical structures on the cycle as we do for the path.  Clearly each digit is distributed identically, since $\C_n$ is vertex-transitive, but we do not have a full analogue of Theorem~\ref{ballot-path}.  However, we can observe the following pattern.

\begin{prop}
Let $n\geq 3$ and $1\leq i\leq n$.  Then $(1,r_2,\dots,r_n)$ is an arithmetical $r$-structure on $\C_n$ if and only if $(1,r_2,\dots,r_n,1)$ is an arithmetical $r$-structure on $\P_{n+1}$.  In particular, the number of arithmetical $d$-structures on $\C_n$ with $d_i=1$ is the Catalan number $C_{n-1}$.
\end{prop}
\begin{proof}
The equivalence follows immediately from the characterizations of arithmetical $r$-structures on paths and cycles (respectively Corollary~\ref{path-characterization} and Proposition~\ref{cycle-characterization}), and the enumeration then follows from Theorem~\ref{ballot-path}.
\end{proof}

\bibliographystyle{amsalpha}
\bibliography{biblio}
\end{document}